\documentclass[12pt,a4paper]{article}
\usepackage{amsmath}
\usepackage{amssymb}

\usepackage{mathtools}
\usepackage{amsthm}
\usepackage[usenames]{color}

\usepackage{graphicx}

\newcommand{\eps}{\varepsilon}
\newcommand{\Z}{{\mathbb Z}}
\newcommand{\N}{{\mathbb N}}

\newcommand{\B}{{\mathsf B}}

\newcommand{\R}{{\mathbb R}}

\newcommand{\s}{{\widehat S}}

\let\phi=\varphi

\newcommand{\E}{{\mathbb E}}

\newcommand{\tW}{{\widetilde W}}

\newcommand{\8}{{\infty}}

\newcommand{\IP}{{\mathbb P}}
\newcommand{\IE}{{\mathbb E}}

\DeclareMathSymbol{\widehatsym}{\mathord}{largesymbols}{"62}

\newcommand{\hW}{\widehat{W}}

\newcommand{\capa}{\mathop{\mathrm{cap}}}

\newcommand{\htau}{\widehat{\tau}}

\newcommand{\tS}{{\widetilde S}}

\allowdisplaybreaks

\newtheorem{theo}{Theorem}
\newtheorem{lem}[theo]{Lemma}
\newtheorem{df}[theo]{Definition}

\newtheorem{rem}[theo]{Remark}

\title{Rate of escape of the conditioned 
   two-dimensional simple random walk}
\author{Orph\'ee Collin$^{1}$ \and
 Serguei~Popov$^{2}$}

\begin{document}

\maketitle

{\footnotesize 
\noindent $^{~1}$Universit\'e Paris Cit\'e 
and Sorbonne Universit\'e, CNRS, Laboratoire de Probabilit\'es,
Statistique et Mod\'elisation, F--75013 Paris, France
\\
\noindent e-mail:
\texttt{collin@lpsm.paris}

\noindent $^{~2}$Centro de Matem\'atica, University of 
Porto, Porto, Portugal\\
\noindent e-mail: \texttt{serguei.popov@fc.up.pt}

}

\begin{abstract}
We prove sharp asymptotic estimates for the rate of escape of the two-dimensional simple random walk conditioned to avoid a fixed finite set. We derive it from asymptotics available for the continuous analogue of this process~\cite{CC22}, with the help of a KMT-type coupling adapted to this setup.
\\[.3cm]\textbf{Keywords:} Brownian motion, conditioning, KMT approximation, transience, Doob's $h$-transform 
\\[.3cm]\textbf{AMS 2010 subject classifications:}
60J10, 60J60, 60J65
\\[.3cm] This work is licensed under a Creative Commons Attribution | 4.0 International licence (CC-BY 4.0, https://creativecommons.org/licenses/by/4.0/).
\end{abstract}

\section{Introduction and main results}
\label{s_intro}
In this paper we study two-dimensional simple random walk (also 
abbreviated as SRW)
conditioned on not hitting the origin. It appeared in~\cite{CPV16}
(see also~\cite{CP22_01,CP17})
as the main ingredient in the construction of two-dimensional
random interlacement process. Besides taking part in that
construction, it has a number of remarkable properties
itself (cf.\ e.g.~\cite{CPV16,GPV19,PRU20}).
Just to mention some of these
(in the following, the conditioned SRW is denoted by~$\s$):
\begin{itemize}
 \item the walk~$\s$ is transient; however, \emph{any}
infinite subset of~$\Z^2$ is recurrent (i.e., visited infinitely
many times a.s.);
 \item a (distant) site~$y$ is eventually visited by~$\s$
with probability converging to~$\frac{1}{2}$, as $|y|\to\infty$.
 \item if one considers a ``typical'' large subset of~$\Z^2$
(i.e., a box or a segment), then the proportion of its sites
which are eventually visited by~$\s$ is a random variable
that converges to Uniform$[0,1]$ in distribution
(as the size of that subset grows).
 \item two independent copies of~$\s$ a.s.\ meet, and, moreover,
 $\s$ a.s.\ meets an independent SRW~$S$ infinitely many times.
\end{itemize}
Here, our aim is to study
the rate of escape of $(\s_n)_{n\geq 0}$ to infinity, 
which we measure via the so-called \emph{future minima process}:
\begin{equation}
\label{df_future_min}
M_n =\inf_{m\geq n} |\s_m|, \qquad n\geq 0,
\end{equation}
where~$|\cdot|$ denotes the Euclidean norm (in~$\Z^2$ or~$\R^2$).
One of the main results of~\cite{PRU20} is that~$M$
``oscillates a lot'': for
every $0 < \delta < \frac{1}{2}$ we have, almost surely,
\begin{equation}
\label{result_PRU}
M_{n} \leq n^{\delta }\ \text{i.o.,} \quad \text{but} \quad M_{n}
\geq \tfrac{\sqrt{n}}{\ln ^{\delta } n}\ \text{i.o.}.
\end{equation}
We aim at obtaining a finer version of the above result,
via a comparison to the corresponding continuous model
(which we define and discuss later).

Let us now pass to the formal discussion.
Let $S=(S_n)_{n\geq 0}$ be the simple random walk on~$\Z^2$.
In the sequel, we will write~$\IP_x$ and~$\IE_x$ for the 
probability and the expectation with respect to the process
(SRW or another one)
started at~$x$, when no confusion can arise. 
The SRW's \emph{potential kernel} 
is the unique function (up to a constant factor) 
on~$\Z^2$ that is~$0$ at the origin and harmonic 
on $\Z^2\setminus \{0\}$. It is defined by:
\[
a(x)= 
\sum_{n\geq 0} \big(\mathbb{P}_0[S_n=0]-\mathbb{P}_0[S_n=x]\big),
  \qquad x\in\Z^2.
\]
Harmonicity of $a$ outside the origin implies that 
$(a(S_{n\wedge \tau_0}))_{n\geq 0}$, where~$\tau_0$ 
is the time of the first visit of~$S$ to the origin, 
is a martingale (we therefore say that~$a$ is a scale function 
of $S_{\cdot\wedge \tau_0}$). 
See Section~4.4 of~\cite{LL10} for more details
on the potential kernel. 
Here, we only recall a useful asymptotic expression for~$a$:
\begin{equation}
\label{asympta}
\frac{\pi a(x)}{2} 
=\ln |x| + \gamma_{\rm EM}+ \frac32\ln 2 + O(|x|^{-2}),
\end{equation}
as $|x|$ goes to infinity, where~$\gamma_{\rm EM}$ is 
the Euler-Mascheroni constant\footnote{$\gamma_{\rm EM}
=\lim_{n\to\infty}\big(1+\frac{1}{2}
+\cdots+\frac{1}{n}-\ln n\big)\approx 0.577$}.

We now consider another Markov chain on $\Z^2\setminus \{0\}$, 
denoted by~$\s=(\s_n)_{n\geq 0}$, 
which is obtained by applying the 
Doob's transform to~$(S_n)_{n\geq 0}$ with respect to
the potential kernel~$a$.
Explicitly, it is a Markov chain on~$\Z^2\setminus \{0\}$ 
with transition probabilities:
\[
p_{x,y}=\frac{a(y)}{4a(x)} \mathbf{1}_{|x-y|=1}.
\]
This Markov chain is called \emph{SRW
on~$\Z^2$ conditioned on not hitting the origin} 
since it appears 
as a natural limiting process, as~$r$ goes to infinity, 
of the law of~$S$ conditioned on reaching distance to the origin
larger than $r$, before visiting the origin.
We have already mentioned some remarkable properties of~$\s$;
let us refer to Section~4.1 of~\cite{P21} for an elementary
introduction to Doob's $h$-transforms, and to 
Section~4.2 of~\cite{P21} for the detailed discussion
of this definition and for derivation of some 
basic properties of~$\s$.

The main result of this paper is the following 
theorem about the future minimum process~$M$
(recall~\eqref{df_future_min}). It can be seen
as the discrete counterpart of the corresponding results 
for the analogous
``continuous model'' (which is the two-dimensional
Brownian motion conditioned on staying outside of a disk,
as explained in Section~\ref{s_coupl_BM} below), 
see Theorems~1.1 and~1.2 of~\cite{CC22}.
\begin{theo} 
\label{th-main}
Independently of the starting point of~$\s$ in $\Z^2\setminus \{0\}$,
the two following results hold. 
\begin{itemize}
 \item For $g: \R_+\to \R_+$  non-increasing such 
 that $t\mapsto (\ln t) g(\ln \ln  t)$ is non-decreasing,
\[
\IP\big[M_n \leq e^{(\ln n) g(\ln \ln n)}\text{ i.o.}\big]
 = 0 {\rm\ or\ } 1
    \text{ according to } \int^\8 g(u) du < \8  {\rm\ or\ }= \8.
\]
 \item The limit 
\[
 K = \limsup_{t \to \8} \frac{M_n}{\sqrt{n \ln \ln \ln  n}} 
\]
exists, is a.s. constant and non-degenerate: $0<K<\infty$.
\end{itemize}
\end{theo}
It is straightforward to see that the above
result indeed generalizes~\eqref{result_PRU};
in particular, take $g\equiv \delta$
to obtain that $M_{n} \leq n^{\delta }$ i.o..
The main idea of the proof will be using 
 the dyadic coupling (also known as KMT 
approximation, or Hungarian embedding) to be able to
``compare'' the conditioned SRW to the conditioned Brownian
motion; then, it will be possible to derive
Theorem~\ref{th-main} from the corresponding
results in the continuous case  
(see Theorem~\ref{th-CC22} in Section~\ref{s_coupl_BM}).
We mention that the actual value of~$K$ is not known.

\subsection{A generalization: SRW conditioned on not hitting 
a finite set}
\label{s_nothit_A}
A natural question that one may pose is the following one:
what if, instead of conditioning the SRW on not hitting
the origin, we condition it on not hitting some \emph{finite}
set $A\subset\Z^2$, will it change anything? In this section
we argue that the answer to this question is essentially ``no''.
While one may take this fact as evident since (assuming without loss
of generality that $0\in A$) $S$ conditioned on not hitting~$A$
``should be'' $\s$ conditioned on not hitting~$A$ and a transient
walk (such as~$\s$) will eventually escape from~$A$ anyway, 
we still decided
to discuss this question here since conditioning on zero-probability
events can be tricky, so, in principle, one should not always
blindly believe one's intuition in that respect.

We also remark that conditioning the SRW on not hitting
an \emph{infinite} set~$A'$ can indeed be a quite different
story (in the sense that the resulting process will be 
``drastically different''\footnote{Indeed, as we mentioned,
$\s$ hits any fixed infinite set infinitely many
times a.s.} from~$\s$); however, 
questions of this sort are beyond the scope of this paper.

Let us fix a finite set~$A$ containing the origin
 and study the SRW conditioned 
on not hitting $A$, which we define now. 
Consider the function $q_A:\Z^2 \to \R$ defined by 
\[
q_A(x)= a(x)-\E_x[a(S_{\tau_A})], \qquad x\in \Z^2,
\]
with $\E_x$ denoting the expectation under 
law $\mathbb{P}_x$ and~$\tau_A$ the hitting time of~$A$ by~$S$.
Function~$q_A$ vanishes on~$A$ and is harmonic
on $\Z^2\setminus A$, so~$q_A$ is a scale function 
for $S_{\cdot\wedge \tau_A}$.
Furthermore, it has the following asymptotic expression:
\begin{equation}
\label{asympta2}
\frac{\pi q_A(x)}2=\ln |x| -\frac\pi2\capa(A) 
+ \gamma_{\rm EM}+ \frac32\ln 2 + O(|x|^{-1}),
\end{equation}
as~$|x|$ goes to infinity, where~$\capa(A)$ is the \emph{capacity} 
of~$A$. We refer to chapter~3 of~\cite{P21} for all necessary 
details.

In the following, we use the general construction 
of chapter~4.1 of \cite{P21} (see also exercise 4.6 there).
Consider the Markov chain on $\Z^2\setminus A$, 
denoted by~$\s^{(A)}=(\s^{(A)}_n)_{n\geq 0}$, 
which is obtained by applying the 
Doob's transform to $(S_n)_{n\geq 0}$ with respect to $q_A$.
Analogously to~$\s$, it is the Markov chain on $\Z^2\setminus A$ 
with transition probabilities:
\[
p^{(A)}_{x,y}=\frac{q_A(y)}{4 q_A(x)} \mathbf{1}_{|x-y|=1}.
\]

We now prove the following fact ($\htau$ denotes the hitting 
times with respect to the conditioned walk~$\s$):
\begin{lem}
\label{l_hS=hSA}
 For any $x_0\notin A$ and any nearest-neighbour path
$\gamma = (\gamma_0,\gamma_1,\ldots,\gamma_m)$ which does not touch~$A$ and such that $x_0=\gamma_0$,
we have
\begin{equation}
\label{eq_hS=hSA}
 \IP_{x_0}[\s_1=\gamma_1,\ldots,\s_m=\gamma_m 
   \mid \htau_A=\infty]
    = \IP_{x_0}[\s^{(A)}_1=\gamma_1,\ldots,
        \s^{(A)}_m=\gamma_m];
\end{equation}
i.e., it holds indeed that~$\s$ conditioned on never hitting~$A$
is~$\s^{(A)}$.
\end{lem}

\begin{proof}
 Let us first show that, for $x_0\notin A$,
\begin{equation}
\label{escape_A_hS}
 \IP_{x_0}[\htau_A=\infty] = \frac{q_A(x_0)}{a(x_0)}.
\end{equation}
Indeed, take a large $R$ such that\footnote{Here and in the 
sequel, $B(x,r)$ stands for the disk 
$\{y\in\mathcal{X}: |y-x|\le r\}$ and $B(x, r)^\complement$ stands for its complement in $\mathcal{X}$, where $\mathcal{X}=\Z^2$
or~$\R^2$ (depending on the context).}
$B(0,R)$ contains 
both~$x_0$ and~$A$, and
abbreviate~$\tau_R$ 
(respectively, $\htau_R$)
to be the hitting time of $B(0,R)^\complement$
by~$S$ (respectively, by~$\s$). Denote $b:=1+\max_{x\in A} |x|$.
By, first, Lemma~4.4 of~\cite{P21},
and, then, Lemma~3.12 and formula~(3.52) of~\cite{P21},
we can write (recall that $0\in A$)
\begin{align*}
 \IP_{x_0}[\htau_A>\htau_R] &= 
   \IP_{x_0}[\tau_A>\tau_R \mid \tau_0>\tau_R] 
   \big(1+O((R\ln R)^{-1})\big)\\
 &= \frac{\IP_{x_0}[\tau_A>\tau_R]}
  {\IP_{x_0}[\tau_0>\tau_R]}\big(1+O((R\ln R)^{-1})\big)\\
  &= \frac{q_A(x_0)}{a(R)+O(\ln b)+O(R^{-1})}\times
   \frac{a(R)+O(R^{-1})}{a(x_0)}\big(1+O((R\ln R)^{-1})\big);
\end{align*}
sending~$R$ to infinity, we obtain~\eqref{escape_A_hS}.

Next, denoting $y=\gamma_m$ we have, 
using~\eqref{escape_A_hS} twice
\begin{align*}
\lefteqn{
\IP_{x_0}[\s_1=\gamma_1,\ldots,\s_m=\gamma_m 
   \mid \htau_A=\infty]
}\\
   &= \frac{a(x_0)}{q_A(x_0)}
   \IP_{x_0}[\s_1=\gamma_1,\ldots,\s_m=\gamma_m 
   , \htau_A=\infty]\\
& = \frac{a(x_0)}{q_A(x_0)} \times
 \IP_{y}[\htau_A=\infty] \times \frac{a(y)}{a(x_0)}
   \IP_{x_0}[S_1=\gamma_1,\ldots,S_m=\gamma_m ]\\
&= \frac{q_A(y)}{q_A(x_0)}
   \IP_{x_0}[S_1=\gamma_1,\ldots,S_m=\gamma_m ]\\
&= \IP_{x_0}[\s^{(A)}_1=\gamma_1,\ldots,
     \s^{(A)}_m=\gamma_m],
\end{align*}
and the proof is complete.
\end{proof}

Now, with Lemma~\ref{l_hS=hSA} to hand, it becomes
clear that Theorem~\ref{th-main} should hold
with~$\s^{(A)}$ on the place of~$\s$,
due to the usual excursion 
decomposition argument. Namely, 
an $\s$-trajectory has an a.s.\ finite number of
excursions from the external boundary of~$A$ to~$A$,
and then the ``escape trajectory'' that does not touch~$A$
anymore. That escape trajectory has the law of~$\s^{(A)}$,
which implies the preceding claim.

%

\subsection{Coupling with the conditioned Brownian motion}
\label{s_coupl_BM}
As we mentioned above,
in~\cite{CC22}, the continuous analogue of the process~$\s$ was 
studied and sharp bounds on the rate of escape were obtained.

Let $W=(W_t)_{t\geq 0}$ be the Brownian motion on~$\R^2$ with covariance matrix $\frac12 I_2$, or equivalently\footnote{The matrix $I_2:=\begin{pmatrix}
1 & 0\\
0 & 1
\end{pmatrix}$ is the identity matrix of size 2.}, the image of the standard Brownian motion on $\R^2$ by the time change $t\mapsto \frac{t}2$. 
For any $\rho\in(0, \infty)$, 
the function $x \mapsto \ln (|x|/\rho)$ 
vanishes on $\partial B(0, \rho)$ and is a scale function for~$W$.
The standard Brownian motion on $\R^2$ \emph{conditioned not 
to visit $B(0,\rho)$} is a diffusion on $\R^2\setminus B(0,\rho)$,
denoted $\hW^\rho=(\hW^\rho_t)_{t\geq 0}$, which can be defined 
e.g.\ via its transition kernel~${\hat p}_\rho$:
for $\|x\|>\rho, \|y\|\geq \rho$, 
\begin{equation}
\label{df_hat_p}
 {\hat p}_\rho(t,x,y) 
  = p_0(t,x,y)\frac{\ln (|y|/\rho)}{\ln(|x|/\rho)} ,
\end{equation}
where~$p_0$ denotes
the transition subprobability density of~$W$ killed on hitting
 the disk~$\B(0,\rho)$. 
We refer to~\cite{CP20} for a more detailed discussion. 

Process $\hW^\rho$ is rotationally invariant, 
its norm is itself a one-dimensional diffusion, which exhibits
interesting renewal properties.
This rotational invariance is the key difference from 
the discrete case, which makes it possible to obtain
sharp results on the rate of escape
by only studying that one-dimensional diffusion,
as was done in~\cite{CC22}.
Denoting by
\[
M^\rho(t)=\inf_{s\geq t} |\hW^\rho_s|,\qquad t\ge 0,
\]
the future minima process of $|\hW^\rho|$, 
we derive from the main results of~\cite{CC22} the following

\begin{theo} 
\label{th-CC22}
Let $\rho\in(0, \infty)$. Independently of the starting point 
of $\hW^\rho$ in $\R^2\setminus B(0,\rho)$, 
the two following results hold. 
\begin{itemize}
\item For $g: \R_+ \to \R_+$ non-increasing such that 
$t\mapsto (\ln t) g(\ln \ln  t)$ is non-decreasing, 
we have:
\[
\IP\big[M^\rho(t) \leq e^{(\ln t) g(\ln \ln t)}\text{ i.o.}\big]
= 0 {\rm\ or\ } 1 
\text{ according to } \int^\8 g(u) \, \mathrm{d}u < \8  
{\rm\ or\ }= \8.
\]
 (Here, ``i.o.'' means ``for an unbounded set of $t$'s''.)
\item The limit 
\[
   K=\limsup_{t \to \8} \frac{M^\rho(t)}{\sqrt{t \ln \ln \ln  t}} 
\]
exists, is a.s.\ constant and is non-degenerate: $0<K<\8$. The value of $K$ is independent of $\rho$.
\end{itemize}
\end{theo}

\begin{proof}
The object of study of \cite{CC22} was the \emph{standard} two-dimensional Brownian motion conditioned on not hitting the unit ball $B(0, 1)$. Since the latter process and the process $\hW^1$ are related (in law) only by the time change $t\mapsto \frac{t}2$, it is immediate to derive the results for $\rho=1$ from Theorems 1.1 and 1.2 of \cite{CC22}. To be completely explicit, we note that the value of $K$ in the present theorem and the value of $K^*$ in Theorem 1.2 of \cite{CC22} have a ratio of $\sqrt{2}$, but we also recall that these values are not known. 

Furthermore, a consequence of scaling invariance of the Brownian motion is that for any $\rho\in(0, \infty)$, process~$\hW^\rho$ is equal in law to process $(\rho \hW^1_{t\rho^{-2}})_{t \ge 0}$.
Hence, Theorem~\ref{th-CC22} holds for all $\rho\in(0, \infty)$, and we see that the value of $K$ is independent of $\rho$.
\end{proof}

Next, denote 
$\rho_0=\exp(-\gamma_{\rm EM}-\frac32\ln 2)\approx 0.1985$ 
and write~$\hW$ for $\hW^{\rho_0}$. 
In order to prove Theorem~\ref{th-main}, 
we will construct a coupling of processes~$\hW$ and~$\s$. 
In Sections~\ref{sec-coupling} and~\ref{sec-couplingproof}, 
we are going to show the following
\begin{theo}
\label{th-couplingv1}
There exists a positive constant~$\alpha$ such that for all~$\eps>0$ and all starting points in $\Z^2\setminus \{0\}$ and $\R^2 \setminus B(0,\rho_0)$, there exists a coupling of~$\s$ 
and~$\hW$ started respectively at these points,
such that with probability not smaller than $1-\eps$:
\begin{equation}
\label{eq-thm-boundt}
\big| |\s_t|-|\hW_t| \big|\leq \ln^\alpha t \quad 
\text{ eventually, as } t\to \infty,
\end{equation}
where $(\s_t)_{t\geq 0}$ is the process defined 
by linear interpolation of the sequence $(\s_n)_{n\geq 0}$.
  
Also,
with probability not smaller than $1-\eps$:
\begin{equation}
\label{eq-thm-corsim}
   |\s_t| \sim |\hW_t|, \text{ as } t\to \infty.
\end{equation}
\end{theo}

\begin{rem}
We can be more precise about the value of constant~$\alpha$ 
in Theorem~\ref{th-couplingv1}. 
Indeed, its proof (which is to be found 
in Section~\ref{sec-couplingproof}) shows that any $\alpha>27$
works. However, we stress that we have no reason 
to believe that the above would have to be sharp.
\end{rem}

\begin{rem}\label{rem:rho01}
According to Theorem \ref{th-CC22}, all processes~$\hW^\rho$ have the same large-$t$ behaviour:  this was proved using the Brownian invariance. Due to transience, it is also possible to show this fact essentialy  in the spirit of Section~\ref{s_nothit_A}.
Thus the choice of $\rho_0=\exp(-\gamma_{\rm EM}-\frac32\ln 2)$ is mostly technical. In fact, it comes from the asymptotic expression~\eqref{asympta}: it is indeed the only choice which guarantees that
\begin{equation}
\label{why_rho_0}
\frac{\pi a(x)}{2}=\ln \frac{|x|}{\rho_0} + O(|x|^{-2}),
\end{equation}
as $|x|$ goes to infinity. 
%
%
%
We will see in the proofs that
$\hW=\hW^{\rho_0}$ is the ``right one''
for comparison with~$\s$, since it has almost the same 
transition probabilities between certain concentric ``levels'': this will be explained further in Remark~\ref{rem:rho02}, in Section~\ref{sec-coupling}.
\end{rem}

\section{Sampling $\hW$ and $\s$}
\label{s_sampling_WS}
As a first step towards the construction of the coupling, 
we describe how a realisation of~$\hW$ (respectively, $\s$) 
can be sampled using an i.i.d.\ family of $W$-trajectories (respectively $S$-trajectories).

\subsection{Notations}
For $m\geq 0$, we set $r_m= \rho_0 \, e^m$ and denote 
$C_m=\{x\in \R^2: |x|= r_m\}$ the circle of radius~$r_m$ 
centered at the origin. Consider a trajectory of~$\hW$ 
started at some point $x\in C_h$, 
where $h$ is some positive integer. 
It can be described in terms of excursions between 
the levels $(C_m, m\geq 1)$. 
Indeed, define the sequence of stopping times $(t_k)_{k\geq 0}$ 
and the \emph{level sequence} $(m_k)_{k\geq 0}$ such that:
$t_0 = 0, m_0=h$ and for $k\geq 0$,
\begin{align}
 t_{k+1} & = \inf \{t>t_k : \hW_{t} \in C_{m_k-1}\cup C_{m_k+1} \},\\
 m_{k+1} & = \ln (|\hW_{t_{k+1}}|/\rho_0).
\end{align}
Also, denote by $x_k=\hW_{t_k}, k\geq 0$ the hitting points
of the circles.

We need to be more careful when introducing 
 the discrete counterparts of the above notations.
 Let us define $\Gamma_0=\{0\}$, 
 $\Gamma_1=\{x\in\Z^2 : |x|=1\}$,
 $\Gamma_2=\{x\in\Z^2 : 1<|x|\leq 2\}$,
and, for $m\geq 3$ (note that $r_3\approx 3.9871$), 
\begin{equation*}
\Gamma_m=\{x\in \Z^2 \colon r_m-1<|x|\leq r_m\}.
\end{equation*}
Observe that these sets are disjoint and that, for all $m\geq 1$, any nearest-neighbour path between~$\Gamma_{m-1}$ and~$\Gamma_{m+1}$ has to pass through~$\Gamma_m$. We can then 
decompose any trajectory of $\s$ started from $\xi\in \Gamma_h$,
with some positive integer~$h$, into excursions between 
the levels $(\Gamma_m, m\geq 0)$. 
Define sequences $(\tau_k)_{k\geq 0}$ and $(\mu_k)_{k\geq 0}$ 
such that: $\tau_0 = 0, \mu_0=h$, and for $k\geq 0$,
\begin{align}
 \tau_{k+1} & = \inf \big\{n>\tau_k : 
  \s_n \in \Gamma_{\mu_k-1}\cup \Gamma_{\mu_k+1} \big\}, \\
 \mu_{k+1} & = 
  m \text{ on } \big\{\s_{\tau_{k+1}}\in\Gamma_m\big\}, 
  m=1,2,3,\ldots. 
\end{align}
Finally, denote by $\xi_k=\s_{\tau_k}, k\geq 0$ the hitting points 
of these ``discrete circles''.

For any $k$, 
$(\hW_t)_{t\in(t_k, t_{k+1}]}$
is called an \emph{excursion}
between \emph{levels}~$m_k$ and~$m_{k+1}$,
and~
$(\s_n)_{n\in (\tau_k, \tau_{k+1}]}$
is called 
an excursion between levels~$\mu_k$ and~$\mu_{k+1}$. 
By construction, the level sequences $(m_k)_{k\geq 0}$ and 
$(\mu_k)_{k\geq 0}$ have increments in $\{-1, +1\}$, 
and by transience of~$\hW$ and~$\s$ these sequences tend to infinity.

\subsection{Constructing $\hW$ from $W$}
\label{constrwW}
Take a positive integer~$m$ and any point~$x$ in~$C_m$. 
We consider the first excursion of~$\hW$ started from level~$m$, 
at point~$x$. By definition of~$\hW$ 
(recall Section~\ref{s_coupl_BM}),
we have
\begin{equation}
\label{eq-defwW}
 \mathbb{P}_x\big[(\hW_t)_{0\leq t\leq t_1} \in \,\cdot\, 
  , \hW_{t_1}\in dy\big]
=\frac{\ln (|y|/\rho_0)}{\ln(|x|/\rho_0)}
\, \mathbb{P}_x\big[\,(W_t)_{0\leq t\leq t_1} 
\in \cdot\, , W_{t_1}\in dy\big]
\end{equation}
(defining $t_1$ for process~$W$ in the same way 
 as for process~$\hW$).

On one hand, using that~$W$ has scaling function $h:x\mapsto \ln |x|$,
we have that 
$\mathbb{P}_x\big[W_{t_1} \in C_{m-1}\big]=\mathbb{P}_x\big[W_{t_1} \in C_{m+1}\big]=\frac12$, 
hence~\eqref{eq-defwW} implies that
\begin{align}
 \mathbb{P}_x\big[\hW_{t_1}\in C_{m-1}\big] 
              & =\frac{m-1}{2m},\\
 \mathbb{P}_x\big[\hW_{t_1}\in C_{m+1}\big] 
              & =\frac{m+1}{2m}.
\end{align}
Applying the strong Markov property, we see that the random sequence 
$(m_k)_{k\geq 0}$ follows the law of 
a Markov chain on~$\N$ with transition probabilities 
$\frac{m-1}{2m}$ to the left and~$\frac{m+1}{2m}$ to the right.
In the sequel, we will refer to this Markov chain
as the \emph{conditioned simple random walk in one dimension},
or 1-CSRW for short. The reason for this name is that
it is indeed the simple random walk on~$\Z$ 
conditioned to stay positive,
because it is obtained by applying 
the Doob's transform (with respect to $x\mapsto |x|$) 
to the simple random walk on~$\Z$. 
It can be interpreted, once again, as the limit 
of the simple random walk (started in the positive domain) 
conditioned to visit site~$N$ before the origin, 
as~$N$ goes to infinity, see Section~4.1 of~\cite{P21}. 
We remark also that, as 1-CSRW has the drift of order~$x^{-1}$
at~$x$, it belongs to the class of 
\emph{Lamperti processes}, cf.\ Chapter~3 of~\cite{MPW16}.

On the other hand, \eqref{eq-defwW} implies that, conditionally 
on the sequence $(m_k)_{k\geq 0}$, the excursions of~$\hW$ 
are simply Brownian excursions conditioned on exiting 
the annuli via the corresponding (inner or outer) boundaries. 

Therefore, one way to sample $\hW$ would be to sample the level
sequence $(m_k)_{k\geq 0}$ as a realisation of 
1-CSRW,
and then to sample excursions of~$W$ conditioned on 
hitting the inner or outer circle first, 
according to the level sequence. 
However, we prefer to use the construction described below, 
which has the virtue of avoiding this conditioning of excursions. 
The procedure we use is the usual acceptance-rejection technique 
of generating random variables.

To construct $(\hW_t)_{t\geq0}$, we need the following ingredients:
\begin{itemize}
\item a collection of i.i.d.\ stochastic processes
 $(\tW^{(k,\ell)}_\cdot, k,\ell \geq 0)$,
 which are all two-dimensional Brownian motions with covariance matrix $\frac12 I_2$
 (started at the origin).
\item a collection of i.i.d.\ random variables
 $(U_{k,\ell}, k,\ell \geq 0)$ with Uniform$[0,1]$
 distribution;
\end{itemize}

Suppose that, initially, the process is at $x \in C_h$,
with some positive integer~$h$. 
Now, what we aim to construct is a sequence of excursions 
(as defined above), i.e., a sequence of time moments
$(t_k, k\geq 0)$ with $t_0=0$, and a sequence of integers 
$(m_k, k\geq 0)$ with $m_0=h$ such that 
$\hW_{t_k}\in C_{m_k}$
and such that on 
the interval~$(t_{k},t_{k+1}]$, $\hW$ travels
between the levels $m_k$ and $m_{k+1}$.

So, assume that we have constructed the process
up to time~$t_{k}$ and that $x_k=\hW_{t_{k}}\in C_{m_k}$.
We consider successively for all $\ell$, the trajectory 
$x_k+\tW^{(k+1,\ell)}_\cdot$, 
stopped when it first hits~$C_{m_k-1}\cup C_{m_k+1}$, 
say at time $s_{k+1}^{(\ell)}$.
\begin{itemize}
\item If the trajectory first hits $C_{m_k+1}$, 
 we keep it: we set 
 $m_{k+1}=m_k+1$, $t_{k+1}=t_k+s_{k+1}^{(\ell)}$ 
 and 
 $\hW_{\tau_k+t}=\hW_{\tau_{k}}+\tW^{(k+1,l)}_t$ 
 for 
 $0< t \leq s_{k+1}^{(\ell)}$.
\item If the trajectory first hits $C_{m_k-1}$, 
 \begin{itemize}
 \item provided $U_{k+1,\ell}\leq 1-\frac{2}{m_k+1}$, 
    we keep it: we set $m_{k+1}=m_k-1$, 
    $t_{k+1}=t_k+s_{k+1}^{(\ell)}$,
    and 
   $\hW_{t_k+t}=x_k+\tW^{(k+1,\ell)}_t$ 
   for $0<t \leq s_{k+1}^{(\ell)}$;
 \item otherwise, we discard it, and restart the procedure 
  using the trajectory $\tW^{(k+1,\ell+1)}_\cdot$ 
  together with the next variable $U_{k+1, \ell+1}$.
 \end{itemize}
\end{itemize}
Once an excursion has been accepted (which, a.s., happens eventually),
we stop the iteration (in $\ell$) and construct the next excursion
starting from $x_{k+1}=\hW_{t_{k+1}} \in C_{m_{k+1}}$.

This construction relies on two key observations:
\begin{itemize}
 \item the sequence of jumps $(m_k, k\geq 0)$ 
 follows the law of 
 1-CSRW.
 Indeed, assume that $m_k=m$ and denote by~$p_m$ 
 (respectively,~$q_m$) the probability that the process 
 we constructed jumps to~$m+1$ (respectively, to~$m-1$). 
 As already observed, using the harmonicity of
$x\mapsto \ln (|x|/\rho_0)$, 
 for each~$\ell$, the trajectory $x_k+\tW^{(k+1,\ell)}_\cdot$ 
 hits~$C_{m+1}$ before~$C_{m-1}$ 
 with probability~$1/2$. 
 Therefore, $p_m$ and~$q_m$ are proportional 
 to, respectively,~$\frac12$ and~$\frac12 (1-\frac{2}{m+1})$. 
 Using the equality $p_m+q_m=1$, 
 we obtain the correct probabilities. 
 \item by construction, the excursions are excursions 
 of the Brownian motion~$W$, as required.
\end{itemize}

\subsection{Constructing $\s$ from $S$}
We are going to present a similar construction for process~$\s$, 
but let us first investigate the decomposition into excursions,
as we did for the continuous process.
Take a positive integer~$m$ and any point~$x$ in~$\Gamma_m$. 
We consider the first excursion of~$\s$ started from level~$m$, 
at point~$x$. 
By definition of~$\s$ as the Doob's transform of~$S$, we have: 
\begin{equation}
\label{eq-defs}
\mathbb{P}_x\big[(\s_n)_{0\leq n\leq \tau_1} \in \,\cdot\, 
  , \s_{\tau_1}=y\big]
= \frac{a(y)}{a(x)}
\, \mathbb{P}_x[\,(S_n)_{0\leq n\leq \tau_1} 
  \in \cdot\, , S_{\tau_1}=y]
\end{equation}
(defining $\tau_1$ for process~$S$ in the same manner 
as for process~$\s$).
In other terms, the law of~$\s_{\tau_1}$ is given by:
\begin{equation}
\label{eq-hitpoints}
\mathbb{P}_x\big[\s_{\tau_1}=y\big] 
 =\frac{a(y)}{a(x)}
   \, \mathbb{P}_x[S_{\tau_1}=y]\, ,
\end{equation}
and, conditionally on $\{\s_{\tau_1}=y\}$, 
$(\s_n)_{0\leq n\leq \tau_1}$ is an excursion of~$S$ 
started from~$x$ and conditioned to 
hit $\Gamma_{m-1}\cup \Gamma_{m+1}$ at~$y$.

To sample a realisation of~$\s$ started at some initial 
point~$\xi\in\Gamma_h$, with some positive integer~$h$, 
we use the following ingredients:
\begin{itemize}
 \item a collection of i.i.d.\ stochastic processes
 $(\tS^{(k,\ell)}_\cdot, k,\ell \geq 0)$,
 which are two-dimensional simple random walks 
 (started at the origin).
 \item a collection of i.i.d.\ random variables
 $(U_{k,\ell}, k,\ell \geq 0)$ with Uniform$[0,1]$
 distribution.\footnote{We use the same notation as in
Section~\ref{constrwW} for these random variables
because they are actually the same random variables; i.e.,
we will couple~$\s$ and~$\hW$ by using the same collection
of~$U$'s.}
\end{itemize}
We will also use the following notation: 
for $m\geq 1$ and $y \in \Gamma_{m-1}\cup \Gamma_{m+1}$, 
define
\begin{equation}
\label{eq-defpi}
\pi_{m,y}
=\frac{a(y)}{\max_{z\in\Gamma_{m-1}\cup\Gamma_{m+1}} a(z)},
\end{equation}
which is a real number in $[0,1]$.

We proceed as before, i.e., we construct recursively a sequence 
of stopping times
$(\tau_k, k\geq 0)$ with~$\tau_0=0$, 
and a sequence of integers $(\mu_k, k\geq 0)$ with~$\mu_0=h$ 
such that 
$\s_{\tau_k}\in \Gamma_{\mu_k}$
and such that on 
the interval~$(\tau_{k},\tau_{k+1}]$, $\s$ travels 
between the two levels $\Gamma_{\mu_{k}}$ and $\Gamma_{\mu_{k+1}}$. 
Recall that the sequence of hitting points is denoted by 
$(\xi_k)_{k\geq 0}$, so, in particular, $\xi_0=\xi$.

Now, assume that we have constructed the process
up to time~$\tau_{k}$ 
and that $\xi_k=\s_{\tau_{k}}\in \Gamma_{\mu_k}$.
Successively, for all values of~$\ell$, 
we consider the trajectory $\xi_k+\tS^{(k+1,\ell)}_\cdot$, 
stopped when it first hits~$\Gamma_{\mu_k-1}\cup \Gamma_{\mu_k+1}$, 
say at (integer) time $\sigma_{k+1}^{(\ell)}$, 
and denote the hitting point by $\xi_{k+1}^{(\ell)}$.
Then we do the following:
\begin{itemize}
 \item if $U_{k+1,\ell}\leq \pi_{m_k,  \xi_{k+1}^{(\ell)}}$, 
 we keep the trajectory: 
 we set $\xi_{k+1}=\xi_{k+1}^{(\ell)}$, $\mu_{k+1}=\mu_k+1$ 
 or $\mu_k-1$ according to $\xi_{k+1}^{(\ell)}$, 
 $\tau_{k+1}=\tau_k+\sigma_{k+1}^{(\ell)}$ 
 and 
 $\s_{\tau_k+n}=\xi_k+\tS^{(k+1,\ell)}_n$ 
 for $n=1, \dots, \sigma_{k+1}^{(\ell)}$.
 \item otherwise, we discard it, and restart the procedure 
 using the next trajectory $\tS^{(k+1,\ell+1)}_\cdot$ 
 and the next variable $U_{k+1, \ell+1}$.
\end{itemize}
Once an excursion has been accepted 
(which eventually happens a.s.), 
we stop the iteration (in~$\ell$) 
and construct the next excursion starting from 
$\xi_{k+1}=\s_{\tau_{k+1}} \in \Gamma_{\mu_{k+1}}$.

The construction we described indeed gives rise
to process~$\s$ since
\begin{itemize}
\item the sequence of hitting points $(\xi_k)_{k\geq 0}$ 
obeys the law described in~\eqref{eq-hitpoints}. 
Indeed, assume that $\xi_k=x$;
then the probability that the process we constructed satisfies 
$\xi_{k+1}=y$ is proportional to 
$\pi_{m, y} \mathbb{P}_x[S_{\tau_1}=y]$, 
hence it is equal to 
$\frac{a(y)}{a(x)} \, \mathbb{P}_x[S_{\tau_1}=y]$;
\item by construction, the excursions are SRW's excursions,
as required.
\end{itemize}

\section{Construction of the coupling}
\label{sec-coupling}
In this section, we will present the construction of 
the coupling mentioned in Theorem~\ref{th-couplingv1}. 
But first, let us make an observation on the range 
of initial points we have to handle. 
Fix some $\eps \in (0,1)$, some $x\in \R^2\setminus B(0,\rho_0)$,
and some $\xi \in \Z^2 \setminus \{0\}$. 
Assume that we want to build the coupling 
with~$\hW$ started from~$x$ and~$\s$ started from~$\xi$. 
Choose some (large)~$h$ such that $r_h \geq |x|\vee |\xi|$, 
then let~$\s$ and~$\hW$ evolve independently 
under their respective laws until times
\begin{align}
t_0 & =\inf\{t\geq 0: \hW_t\in C_h\},\\
\tau_0 & =\inf\{n\geq 0 : \s_n \in \Gamma_h\}.
\end{align}
Then, we observe that
\begin{align}
\big||\s_t|-|\hW_t|\big|
& \leq \big||\s_t|-|\s_{t+\tau_0-t_0}|\big|
 +\big||\s_{\tau_0+t-t_0}|-|\hW_t|\big| \nonumber \\
& \leq |\tau_0-t_0|+\big||\s_{\tau_0+t-t_0}|-|\hW_t|\big|,
\label{eq-initialpoints}
\end{align}
where we have used in the last inequality that~$\s$ 
makes jumps of size~$1$.
Now, if we manage to construct~$\s$ and~$\hW$ 
after times~$\tau_0$ and~$t_0$ such that 
with probability not smaller than $1-\eps$
\begin{equation}
\label{eq-afterCm}
\big||\s_{\tau_0+t}|-|\hW_{t_0+t}|\big|\leq \ln^\alpha t 
\quad \text{eventually},
\end{equation}
then we will have eventually, plugging this bound 
into~\eqref{eq-initialpoints}, that
\begin{equation}
\big||\s_t|-|\hW_t|\big|\leq |\tau_0-t_0|
  + \ln^\alpha (t-\tau_0),
\end{equation}
which, for any $\alpha'>\alpha$, is eventually not larger 
than $\ln^{\alpha'} t$.

Hence we are left with the task of obtaining~\eqref{eq-afterCm},
which means constructing the coupling 
starting from any prescribed points~$\xi\in \Gamma_h$ 
and~$x \in C_h$, with~$h$ chosen as large as we desire.

In~\cite{KMT75}, J.~Koml{\' o}s, P.~Major, and G.~Tusn{\' a}dy 
proved the existence of a coupling between the one-dimensional simple random 
walk, say $(X_n)_{n\geq 0}$,
and the one-dimensional Brownian motion, say $(B_t)_{t\geq 0}$, 
both started at the origin, and positive constants~$C', K', \lambda'$ 
such that for all $n\in\N$ and~$x>0$
\begin{equation}
\label{eq-KMT1D}
\IP\big[\max_{k=1, \dots, n} |X_k-B_k|\geq C' \ln n + x\big]
 \leq K' e^{-\lambda' x}.
\end{equation}
We refer to Chapter~7 of~\cite{LL10} 
for a modern introduction to this coupling, 
which is known under the denomination of \emph{dyadic coupling}, 
\emph{KMT approximation}, or \emph{Hungarian embedding}, and in particular to Section 7.6 of~\cite{LL10} for an extension to the two-dimensional case. Indeed, using a classical representation of the two-dimensional simple random walk $(S_n)_{n\geq 0}$ as a function of two independent one-dimensional simple random walks, it is possible to build a coupling of the processes  $(S_n)_{n\geq 0}$ and $(W_t)_{t\geq 0}$ such that:
\begin{equation}
\label{eq-KMT}
\IP\big[\max_{k=1, \dots, n} |S_k-W_k|\geq C \ln n + x\big]
 \leq K e^{-\lambda x}.
\end{equation}
with some positive constants $C, K, \lambda$. To illustrate the strength of the KMT coupling, let us underline that, using the Borel-Cantelli lemma, 
it is immediate to get that for any $\eps>0$, 
almost surely, $|S_n-W_n|\leq (C+1/\lambda+\eps)\ln n$ 
eventually. 
Using this two-dimensional KMT coupling, 
we are going to construct a 
coupling between~$\s$, 
the simple random walk conditioned on not hitting the origin, 
and $\hW$, 
the Brownian motion (with covariance matrix $\frac12 I_2$) conditioned on not hitting 
the ball $B(0,\rho_0)$. 

Our aim is to achieve in parallel the constructions 
of~$\hW$ and~$\s$ described in previous section, 
by using KMT-coupled trajectories of~$W$ and~$S$, 
and common variables $(U_{k,\ell}, k,\ell \geq 0)$, 
in such a way that the moduli stay close.

\begin{rem}\label{rem:rho02}
Following Remark \ref{rem:rho01}, we stress that $\rho_0$ has been chosen in such a way that the acceptance-rejectance thresholds in the constructions of $\hW$ and $\s$ (almost) coincide. Indeed, recalling \eqref{why_rho_0}, we have :
\begin{equation*}
\sup_{\xi\in\Gamma_{m+1}} 1-\pi_{m, \xi}\!
=\!O\!\left(\!\frac1{m e^m}\!\right) \quad 
\text{and} \quad \sup_{\xi\in\Gamma_{m-1}} 
\left|\left(1-\frac2{m+1}\right)-\pi_{m, \xi}\right|
\!=\!O\!\left(\!\frac1{m e^m}\!\right),
\end{equation*} 
while for a different choice of $\rho_0$, the second supremum would instead be only $O\left(\frac1{m^2}\right)$.
\end{rem}

To construct the coupling, we need the following ingredients:
\begin{itemize}
\item a collection of i.i.d.\ pairs of stochastic processes
 $\big((\tW^{(k,\ell)}_\cdot,\tS^{(k,\ell)}_\cdot), 
   k,\ell \geq 0\big)$,
where~$\tW$ is a two-dimensional Brownian motion with covariance matrix $\frac12 I_2$
(started at the origin),
 $\tS$ is a two-dimensional SRW (also started at the origin), and 
 they are KMT-coupled (within each pair);
\item a collection of i.i.d.\ random variables
 $(U_{k,\ell}, k,\ell \geq 0)$ with Uniform$[0,1]$
 distribution.
\end{itemize}

Suppose that, initially, both processes are at level~$h$, 
with~$h$ large, 
i.e., $x_0=\hW_0 \in C_h$ and $\xi_0=\s_0\in \Gamma_h$.
What we intend to construct are two sequences of time moments
$(t_k, \tau_k, k\geq 0)$, such that 
$\hW_{t_k}\in C_{m_k}, \s_{\tau_k} \in \Gamma_{\mu_k}$, 
and with probability not smaller than $1-\eps$
the following happens:
for all $k\geq 0$, $\hW$ on 
the interval~$(t_k,t_{k+1}]$ and~$\s$ on the 
interval~$(\tau_k,\tau_{k+1}]$ ``travel together''
between the same levels (i.e., the level sequences~$m$ and~$\mu$
coincide) and the difference $|t_k-\tau_k|$
does not grow too much.

So, assume that we have constructed both processes
up to~$t_k$ and~$\tau_k$, and both 
$x_k=\tW_{t_k}$ and~$\xi_k=\s_{\tau_k}$ belong to~$C_{m_k}$ 
and, respectively, to $\Gamma_{m_k}$.

We now intend to simultaneously construct~$\hW$ on 
the interval~$(t_k,t_{k+1}]$ and~$\s$ on the 
interval~$(\tau_k, \tau_{k+1}]$; we will do that 
using the $U$'s and the $(\tW,\tS)$'s with the first
index~$k$. We successively try $\ell=1,2,3,\ldots$,
and each step will be a \emph{success} (in which case
we have what we wanted, and can pass from~$k$ to~$k+1$), 
a \emph{failure}
(this means that the candidate excursions were discarded, 
so we pass to the next~$\ell$ and try to obtain a success),
or a \emph{catastrophe} 
(which means that the processes have really 
decoupled).
A technical operation is needed, 
which is to rotate the trajectory $\tW^{(k+1,\ell)}_\cdot$. 
Take~$\theta_k$ some angle so that~$\gamma_{\theta_k}$, 
the rotation by angle~$\theta_k$, 
sends~$x_k$ near~$\xi_k$; 
more precisely, we require that~$|\xi_k-\gamma_{\theta_k}(x_k)|\leq 1$. 
Denote by~$s_{k+1}^{(\ell)}$ the hitting time
of~$C_{m-1}\cup C_{m+1}$
by $x_k+r_{-\theta_k}\big(\tW^{(k+1,\ell)}_\cdot\big)$,
and by~$\sigma_k^{(\ell)}$ the hitting time
of~$\Gamma_{m-1}\cup \Gamma_{m+1}$
by $\xi_k+\tS^{(k+1,\ell)}_\cdot$,
and let $x_{k+1}^{(\ell)}$ and $\xi_{k+1}^{(\ell)}$ 
be the corresponding hitting points.
At the~$\ell$-th try, we declare it to be a 
\begin{itemize}
 \item \emph{success}, if the four following assumptions hold
 \begin{itemize}
  \item (Assumption~1) the KMT coupling event occurred 
up to time~$r_m^3$, and $\tW^{(k+1,\ell)}$ 
is ``controlled'' between integer times, i.e.,
  \begin{equation}
    \max_{n=1, \dots, r_m^3} 
     \max_{n\leq t\leq n+1}
      |\tS^{(k+1,\ell)}_n-\tW^{(k+1,\ell)}_t|\leq D m,
   \end{equation}
   with some $D>0$ to be chosen later;
  \item (Assumption~2) 
  $s_{k+1}^{(\ell)} , \sigma_{k+1}^{(\ell)} \in [r_m, r_m ^3]$;
  \item (Assumption~3) 
  $|s_{k+1}^{(\ell)}-\sigma_{k+1}^{(\ell)}| \leq m^\beta$ 
  with some $\beta>0$ to be chosen later;
  \item (Assumption~4a) 
  one of these two conditions applies 
  (recall the definition of~$\pi_{m, y}$ 
  in~\eqref{eq-defpi}):
   \begin{itemize}
    \item 
     $x_{k+1}^{(\ell)} \in C_{m+1}$, $\xi_{k+1}^{(\ell)} 
      \in \Gamma_{m+1}$,
     and 
     $U_{k+1,\ell}\leq \pi_{m_k,\xi_{k+1}^{(\ell)}}$
     (``both went outside and we decided to keep both'');
    \item 
     $x_{k+1}^{(\ell)}\in C_{m-1}$,
     $\xi_{k+1}^{(\ell)} \in \Gamma_{m-1}$
     and 
     $U_{k+1,\ell} \leq \pi_{m_k,\xi_{k+1}^{(\ell)}} 
      \wedge (1-\frac{2}{m+1})$
     (``both went inside and we decided to keep both'');
    \end{itemize}
 \end{itemize}
 \item \emph{failure}, if the above Assumptions~1, 2 and~3 
 hold and the following assumption holds:
     \begin{itemize}
     \item (Assumption~4b) 
      $x_{k+1}^{(\ell)}\in C_{m-1}$,
     $\xi_{k+1}^{(\ell)} \in \Gamma_{m-1}$
     and 
     $U_{k,\ell} \geq \pi_{m_k,\xi_k} \vee (1-\frac{2}{m+1})$
	(``both went inside and we decided to discard both'');
	\end{itemize}
 \item \emph{catastrophe}, otherwise.
\end{itemize}

We proceed as follows. 
In case of failure, we repeat the procedure with the next~$\ell$, 
until we meet a success or a catastrophe. 
In case of success, we declare the sampled excursions
to be the next excursions of our processes~$\hW$ and~$\s$, 
i.e., we set $m_{k+1}=\mu_{k+1}=\ln |x_{k+1}^{(\ell)}|$;
then on one hand,
\[
t_{k+1}=t_k+s_{k+1}^{(\ell)}, \quad 
\text{and} \quad 
\hW_{t_k+t}=x_k+r_{-\theta_k}\left(\tW^{(k+1,\ell)}_t\right), 
\quad \text{for } 0<t \leq s_{k+1}^{(\ell)};
\]
and, on the other hand,
\[
\tau_{k+1}=\tau_k+\sigma_{k+1}^{(\ell)}, \quad 
\text{and} \quad \s_{\tau_k+n}=\xi_k+\tS^{(k+1,\ell)}_n, 
\quad \text{for } n=1, \dots, \sigma_{k+1}^{(\ell)}.
\]
Finally, we denote the exiting points by $x_{k+1}=x_{k+1}^{(\ell)}$
and $\xi_{k+1}=\xi_{k+1}^{(\ell)}$
before moving on to the next step of the coupling 
(constructing the ($k+2$)-th excursions).
Instead, if a catastrophe occurs, we abandon the procedure and 
let both processes~$\hW$ and~$\s$ evolve independently 
(according to their respective laws)
forever after times~$t_k$ and~$\tau_k$. 
However, for formal reasons 
we still define level sequences $(m_k)_{k\ge0}$ 
and $(\mu_k)_{k\ge0}$: almost surely, they will cease 
to coincide at some point.

\section{Proof of Theorem \ref{th-couplingv1}}
\label{sec-couplingproof}
In this section, we will prove that the above-described 
coupling satisfies statements~\eqref{eq-thm-boundt} 
and~\eqref{eq-thm-corsim} of Theorem~\ref{th-couplingv1}. 
We begin with two lemmas which control the probability of 
an occurrence of a catastrophe.

First, we define and estimate the probability of an occurrence of 
a catastrophe in one step of the coupling.
\begin{df}
For $(x, \xi) \in C_m\times \Gamma_m$, 
denote by $\psi_{x,\xi}$ 
the probability that a catastrophe occurs
during the construction of the first excursion 
of the coupling started at $(x, \xi)$. We set 
\[
p_m=\max_{(x, \xi) \in C_m\times \Gamma_m} \psi_{x,\xi}.
\]
\end{df}
Next, recall the above Assumptions~1 and~3, where we left the choice of $D$ and $\beta$ to be made latter. We have
\begin{lem}
\label{lem-sumfinite}
If~$D$ and~$\beta$ are large enough, then
\begin{equation}
\label{eq-sumpm}
\sum_m m^2 p_m < \infty.
\end{equation}
\end{lem}

\begin{proof}
Let us bound $\psi_{x,\xi}$ uniformly in $(x,\xi)$ belonging 
to $C_m\times \Gamma_m$. 
Recall that, in order to build the first excursion of the coupling, we try successively for each $\ell$, the trajectories $(\tW^{(1, \ell)}, \tS^{(1, \ell)}), l\geq 1$ and we stop with the first try which is not a failure.
We observe that the probability that one try in~$\ell$ ends up in a failure is not larger than~$1/2$
(the probability that the Brownian trajectory hits the inner circle). Hence, it is enough to bound the probability 
that a catastrophe occurs in one try in~$\ell$.
Say we look at the first one, we denote $\tS=\tS^{(1,1)}$ and 
$\tW=\tW^{(1,1)}$. 
We handle the different assumptions that appear 
in the definition of a success or a failure. 

But first, let us recall the following facts.
There exist 
real positive numbers~$c$ and~$ c'$ 
such that for all $t\geq 0$ and all $\kappa>0$ we have 
\begin{align}
\mathbb{P}\left[\max_{0\leq s\leq t}|\tW_s|> \kappa\sqrt{t}\right] 
& \leq c \exp(-c' \kappa^2),
\label{max_tW_greater}\\
\mathbb{P}\left[\max_{0\leq s\leq \kappa t}|\tW_s| 
< \sqrt{t}\right] 
& \leq c \exp(-c' \kappa),
\label{max_tW_less}
\end{align}
and the same bounds hold with real times~$t\ge 0$
replaced by integer times~$n$ and 
$\max_{0\leq s\leq t}|\tW_s|$ 
replaced by $\max_{0\leq j\leq n}|\tS_j|$.
These claims are simple consequences of the similar 
statements 
for the one-dimensional Brownian motion and the one-dimensional simple random walk, which are classical results. 
More specifically, in the one-dimensional case, the first statement is a consequence of the reflection principle together with usual large deviation estimates; the second statement follows from the fact that, on each subinterval of length $t$ (regardless of the starting position) the $\max$ is at least $\sqrt{t}$ with uniformly positive probability.
\begin{enumerate}
 \item 
 Pick $\delta>0$.
 The KMT bound \eqref{eq-KMT} gives
\[
\mathbb{P}\left[\max_{n=1, \dots, r_m^3} 
|\tS_n-\tW_n|\geq (3C+\delta) m\right]\leq K\rho_0^{3C\lambda} 
e^{-\lambda \delta m}.
\]
 Furthermore using a union bound and \eqref{max_tW_greater} we have
\begin{align*}
 \mathbb{P}\left[\max_{n=1, \dots, r_m^3} 
  \max_{n\leq t\leq n+1}|\tW_t-\tW_n|\geq \delta m\right]
 & \leq r_m^3 \times 
  \mathbb{P}\left[\sup_{0 \leq t \leq 1} |\tW_t|\geq \delta m\right]\\
 & \leq \rho_0^3 e^{3m} \times c e^{-c'(\delta m)^2}.
\end{align*}
 Hence, for any $D>3C$, 
denoting by~$p^{(1)}_{m}$ the probability that 
 Assumption~1 does not hold,
 we have $\sum_m m^2 p^{(1)}_{m}<\infty$.
 \item We recall that $s_1^{(1)}$ is the hitting time of $C_{m-1}\cup C_{m+1}$ by trajectory $x+\gamma_{-\theta_0}(\tW_{\cdot})$ and we denote it $s_1$ for short. Similarly, we denote $\sigma_1:=\sigma_1^{(1)}$ the hitting time of $\Gamma_{m-1}\cup \Gamma_{m+1}$ by trajectory $\xi+\tS_{\cdot}$. 
 We are going to bound the probability that Assumption~2 
 does not hold, uniformly in $(x, \xi)$. 
 We use the following implications. 
 \begin{itemize}
  \item 
   If $s_1<r_m$ (respectively, if $\sigma_1<r_m$),
   then process~$\tW$ (respectively, process~$\tS$) 
   has left the ball $B(0, r_m/4)$ before time~$r_m$.
  \item 
   If $s_1>r_m^3$ (respectively, if $\sigma_1>r_m^3$),
   then process~$\tW$ (respectively, process~$\tS$) 
   has stayed inside the ball $B(0, 4 r_m)$ up to time $r_m^3$.
  \end{itemize}
Choosing $(t,\kappa)=(r_m,\sqrt{r_m}/4)$ 
in~\eqref{max_tW_greater} 
and $(t,\kappa)=(16r_m^2, r_m/16)$ in~\eqref{max_tW_less}
(and similarly for the process~$\tS$), 
we derive a bound~$p^{(2)}_{m}$ on the probability that
Assumption~2 does not hold, which is uniform in $(x, \xi)$ 
and satisfies $\sum_m m^2 p^{(2)}_{m}<\infty$.  
  
\item Now, we are going to estimate the probability 
that $|s_1-\sigma_1|>m^\beta$.
We work on the event that Assumption~1 holds and also that both excursions go in the outward direction. 
Denote by $T_r$ the hitting time of~$\partial B(0,r)$ 
by process $x+r_{-\theta_0}(\tW_\cdot)$
and consider 
also
the times $T_-:=T_{r_{m+1}-Dm-1}$ and $T_+:=T_{r_{m+1}+Dm}$.
 Obviously, $s_1=T_{r_{m+1}}$ belongs to $[T_-,T_+]$.
On the event $\{T_+\le r_m^3\}$ 
(whose complement has exponentially decaying probability 
arguing as in point~2 
of the present proof), 
Assumption~1 ensures that $\sigma_1$ also belongs to $[T_-,T_+]$. 
Hence, $|s_1-\sigma_1|\le T_+-T_-$ and we are interested in bounding the difference $T_+-T_-$ with large probability.

Recall that $\tW$ is, up a to a time change $t\mapsto t/2$, a standard two-dimensional Brownian motion. Using the strong Markov property, 
it is just a matter of bounding $T^X_{r_{m+1}+Dm}$ 
where $X$ is a Bessel process of order 2 started at $r_{m+1}-Dm-1$ and where $T^X_r$ is the hitting time of $r$ by $X$.
In order to do so, we are going to show that, as soon as~$\delta>0$, 
the term
\[
m^2 \,  \IP_{r_{m+1}-Dm-1}\left[T^X_{r_{m+1}-m^{4+\delta}}<T^X_{r_{m+1}+Dm}\right]
\]
and the term
\[
m^2 \, \IP_{r_{m+1}-Dm-1}\left[ T^X_{r_{m+1}-m^{4+\delta}} 
\wedge T^X_{r_{m+1}+Dm}\ge m^{8+3\delta}\right],
\]
are the general terms of convergent series. Observe that if $T^X_{r_{m+1}+Dm} \ge m^{8+3\delta}$ then at least one of the two events appearing in the above probabilities holds. 

So let us prove the convergence of the two above-mentioned series. The first probability can be computed explicitely using that $r \mapsto \ln r$ is a scale function for $X$ and we obtain that it is $O(m^{-(3+\delta)})$. For the second probability, we consider the stochastic differential equation equation satisfied by $X$, i.e.:
\[dX_t = dB_t +\frac1{2X_t} d t, \qquad t\ge 0,\]
with $B_\cdot$ a standard Brownian motion. In integral form this yields that
\begin{equation}\label{eq-Bes2}
B_t =X_t-X_0-\int_0^t \frac1{2X_s} ds , \qquad t\ge 0.
\end{equation}
We make the following observation for any $0<a<b$ and $T>0$. If $X$ is started inside $[a, b]$ and stays there until time $T$, then \eqref{eq-Bes2} implies that for all $t\in[0, T]$:
\[ a-b - \frac{T}{2a} \le B_t \leq b-a.\]
For our purpose, we take $a=r_{m+1}-m^{4+\delta}$ and $b=r_{m+1}+Dm$ and $T=m^{8+3\delta}$, so we bound the probability we are interested in by the probability that a one-dimensional standard Brownian motion remains in $[a-b -\frac{T}{2a}, b-a]$ up to time $T$ and using the analogue of \eqref{max_tW_less} for the one-dimensional case, this decays as $\exp(-c m^{\delta})$, hence the convergence of the series.

Thus, as soon as $\beta>8$, we get a uniform bound $p^{(3\text{out})}_{m}$ on the probability that Assumption~1 holds, that both excursions go outward, but that Assumption~3 
does not hold, satisfying $\sum_m m^2 p^{(3\text{out})}_{m}<\infty$. 
Similarly, one can obtain, as soon as $\beta>8$, 
a uniform bound $p^{(3\text{in})}_{m}$ on the probability that Assumption~1 holds, that both excursions go inward but that Assumption~3 does not hold, satisfying $\sum_m m^2 p^{(3b)}_{m}<\infty$.
\item Recall the defintion of $\pi_{m, \xi}$ in \eqref{eq-defpi}. As already mentioned in Remark \ref{rem:rho02}, the asymptotic expression~\eqref{asympta}, or \eqref{why_rho_0}, for~$a$ 
yields that
\begin{equation*}
\sup_{\xi\in\Gamma_{m+1}} 1-\pi_{m, \xi}\!
=\!O\!\left(\!\frac1{m e^m}\!\right) \quad 
\text{and} \quad \sup_{\xi\in\Gamma_{m-1}} 
\left|\left(1-\frac2{m+1}\right)-\pi_{m, \xi}\right|
\!=\!O\!\left(\!\frac1{m e^m}\!\right).
\end{equation*} 
 These suprema yield a uniform bound $p^{(4)}_{m}$ on the probability 
 that neither Assumption~4a nor Assumption~4b hold, 
 and we have $\sum_m m^2 p^{(4)}_{m}<\infty$.
\end{enumerate}
 Putting everything together, 
 we have proved the convergence of the series $\sum_m m^2 p_m$, 
 when~$D>3C$ and~$\beta>8$.
\end{proof}

In the following, we assume that~$D$ and~$\beta$ 
have been chosen large enough so that~\eqref{eq-sumpm} holds. 
Recall the discussion in the beginning of 
Section~\ref{sec-coupling}: 
it is enough to construct the coupling with some 
\emph{large} starting level~$h$.

\begin{lem}\label{lem:expectationcatastrophe}
As $h$ increases to infinity, 
the probability that a catastrophe occurs 
in the construction of the coupling converges to~$0$.
\end{lem}

\begin{proof}
Take $(x, \xi)\in C_h\times \Gamma_h$ and consider 
the coupling starting from $(x, \xi)$. 
For all $k\ge 0$, for all $m\ge 1$, conditionally on $\{m_k=m\}$, 
the probability that a catastrophe occurs in the ($k+1$)-th 
step of the construction of the coupling is bounded by~$p_m$. 
Therefore, using a union bound, the probability of an occurrence of a catastrophe 
is bounded by
\[
\sum_{k\geq 0}\sum_{m\geq 1} \mathbb{P}[m_k=m] \times p_m 
=  \sum_{m\geq 1} \mathbb{E}[\# \{k\geq 0 : m_k=m\}] \times p_m.
\]
Now, to estimate the expectation in the above sum, 
recall that $(m_k)_{k\geq 0}$ 
follows the law of 
1-CSRW
(even after a possible catastrophe), and that it is started at $h$ (since $\hW$ is started at $x\in C_h$). 
Using the fact that 
1-CSRW (stopped at the hitting time of $1$) has scaling function~$m\mapsto 1/m$, the optional stopping theorem yields that for $m\le h$, the 1-CSRW started at $h$ hits $m$ with probability is $ \frac{m}{h}.$ When $m>h$, this probability is equal to 1 due to transience. Furthermore, denoting by $E_m$ the expected number of visits at $m$ by the 1-CSRW started at $m$, we have using the Markov property after one step: $E_m=1+ \frac{m-1}{2m} E_m+ \frac{m+1}{2m} \frac{m}{m+1}E_m$, hence $E_m =2m$. Using the Markov property at the hitting time of $m$ by $(m_k)_{k\ge 0}$, we get:
\[ 
\mathbb{E}[\# \{k\geq 0 : m_k=m\}]= \Big(\frac{m}{h}\wedge 1\Big)\times 2m\le \frac{2m^2}{h}.
\]

Since $\sum m^2 p_m<\infty$,
we have derived a bound that converges to 0 as~$h$ goes to infinity.
\end{proof}

We are now ready to prove \eqref{eq-thm-boundt}. 
Consider some fixed $\eps \in (0,1)$. 
Take some~$h$ large enough so that with probability 
not smaller than $1-\eps$, 
the coupling starting from level~$h$ 
(i.e., from any prescribed points $x \in C_h, \xi \in \Gamma_h$) 
sees no catastrophe. From now on, we work on this event.

Consider some integer~$k$ and some time $t\in[t_k, t_{k+1}]$, 
so that~$\hW_t$ belongs to the ($k+1$)-th excursion of~$\hW$.
We use the following bound:
\begin{equation}
\label{eq-boundbytrans}
\big| |\s_t|-|\hW_t|\big| 
\leq \big| |\s_t|-|\s_{t-(t_k-\tau_k)}|\big|
  +\big| |\s_{t-(t_k-\tau_k)}|-|\hW_t|\big|.
\end{equation}
Since $\s$ makes jumps of size 1, 
the first term in the right hand-side of~\eqref{eq-boundbytrans} 
is bounded by $|t_k-\tau_k|$.
For the second term, we proceed by distinguishing cases. 
We abbreviate $t^*=t-(t_k-\tau_k)$. Note that~$t^*$ 
is larger than~$\tau_k$. 
We distinguish cases according to the fact that~$t^*$ 
belongs to the ($k+1$)-th excursion of~$\s$, or not.
\begin{itemize}
\item 
 \emph{case 1: $\tau_k \leq t^* \leq \tau_{k+1}$.} 
 Since~$t$ and~$t^*$ both belong to the ($k+1$)-th 
 excursion of~$\hW$ (respectively, $\s$), 
 and $t^*-\tau_k=t-t_k$, 
 the second term in the right-hand side of~\eqref{eq-boundbytrans} 
 can be controlled using the KMT bound. 
 Recall the role of the angle~$\theta_k$ 
 in the construction of the coupling. 
If~$\ell$ is the index of the try where the first success
happened during the ($k+1$)-th step of the coupling, 
then 
$\s_{t^*}=\s_{\tau_k}+\tS_{t^*-\tau_k}^{(k+1,\ell)}$, 
where $(\tS_t^{(k+1,\ell)})_{t\geq 0}$ 
is the linear interpolation of $(\tS_n^{(k+1,\ell)})_{n\geq 0}$.
Therefore
\begin{align*}
\big| |\s_{t^*}|-|\hW_t|\big|
& = \big| |\s_{t^*}|-|\gamma_{\theta_k}(\hW_t)|\big|\\
& \leq \big| \s_{t^*}-\gamma_{\theta_k}(\hW_t)\big|\\
& \leq \big| \s_{t^*}-\s_{\tau_k}-\gamma_{\theta_k}(\hW_t-\hW_{t_k})\big|
+ \big| \s_{\tau_k}-\gamma_{\theta_k}(\hW_{t_k})\big|\\
& =\big|\tS_{t^*-\tau_k}^{(k+1,\ell)}-\tW_{t-t_k}^{(k+1,\ell)}\big|
+ | \s_{\tau_k}-\gamma_{\theta_k}(\hW_{t_k})\big|\\
& \leq D m_k+ 1.
\end{align*}

\item 
\emph{case 2.} If $t^*>\tau_{k+1}$, we need to adapt this argument.
We can still bound the distance between $\gamma_{\theta_k}(\hW_{t})$ 
and $\s_{\tau_k}+\tS_{t^*-\tau_k}^{(k+1,\ell)}$, 
but this time $\s_{\tau_k}+\tS_{t^*-\tau_k}^{(\ell)}$ 
is (a priori) not equal to $\s_{t^*}$, 
so we also need to bound the difference between them. 
We do that by ``backtracking'' to time $\tau_{k+1}$. 
Observe that
\[
\s_{\tau_{k+1}}
=(\s_{\tau_k}+\tS_{t^*-\tau_k}^{(k+1,\ell)})
+ \tS_{\tau_{k+1}-\tau_k}^{(k+1,\ell)} 
-\tS_{t^*-\tau_k}^{(k+1,\ell)}, 
\]
so that, by the triangle inequality, 
\begin{align*}
\big|\s_{t^*}-(\s_{\tau_k}+\tS_{t^*-\tau_k}^{(k+1,\ell)})\big|
& \leq \big|\s_{t^*}-\s_{\tau_{k+1}}\big|+\big|\s_{\tau_{k+1}}
-(\s_{\tau_k}+\tS_{t^*-\tau_k}^{(k+1,\ell)})\big|\\
& \leq 2(t^*-\tau_{k+1})\\
& \leq 2 |t_{k+1}-\tau_{k+1}|.
\end{align*}
Proceeding as in case 1, we obtain
\begin{equation*}
\big| |\s_{t^*}|-|\hW_t|\big|
 \leq D m_k+1+ 2|t_{k+1}-\tau_{k+1}|.
\end{equation*}
\end{itemize}

Overall, the two terms in the right-hand side 
of~\eqref{eq-boundbytrans} are bounded respectively 
by $|t_k-\tau_k|$ and $Dm_k+1+2|t_{k+1}-\tau_{k+1}|$. 
To conclude the proofs we bound these quantities
when~$t$ is large, i.e., when~$k$ is large.

Recall the notations~$s_k^{(\ell)}$ and~$\sigma_k^{(\ell)}$ 
 from the construction of the coupling. 
If~$\ell$ is the index of the first success 
in the $k$-th step of the construction of the coupling, 
we denote these quantities by~$s_k$ and~$\sigma_k$: they are the respective lengths of the $k$-th excursions of $\hW$ and $\s$. For all $k\geq 1$, we have
\[
 t_k=\sum_{j=1}^k s_j \quad 
 \text{ and } \quad \tau_k=\sum_{j=1}^k \sigma_j.
\]
We recall Assumptions 2 and 3 in the definition of a success. 
We have for all $k\geq 1$
\begin{equation}
\label{eq-bountk}
t_k\geq s_k \geq r_{m_k}= \rho_0 \, e^{m_k},
\end{equation}
so
\[
m_k\leq \ln (t_k/\rho_0) \leq \ln ( t/\rho_0).
\]
Besides, as a property of 
1-CSRW,
almost surely, we have eventually
\[
m_k\geq k^{\frac13}
\]
(in fact, the exponent $\frac13$ can be replaced by 
any exponent in $(0, \frac12)$, 
see for example Theorem~3.2.7 in \cite{MPW16} 
and also Example~3.2.8 there);
hence, if~$k$ is large enough, then
\begin{equation}
\label{eq-boundk}
k\leq m_k^3\leq \ln^3 (t/\rho_0).
\end{equation}
Furthermore, for all $k\ge 1$:
\begin{equation}
\label{eq-boundtk-tauk}
|t_k - \tau_k|\leq \sum_{j=1}^{k} |s_j-\sigma_j|
\leq \sum_{j=0}^{k-1} m_j^\beta
\leq k \max_{j=1, \dots, k} m_j^\beta \leq k (k+h)^\beta 
\leq (k+h)^{\beta +1}.
\end{equation}
Using~\eqref{eq-boundk} and~\eqref{eq-boundtk-tauk}, 
we obtain for large $k$ that 
\[
|t_k-\tau_k|\leq (k+h)^{\beta+1}
\leq (\ln^3 (t/\rho_0)+h)^{\beta+1}.\]
\[|t_{k+1}-\tau_{k+1}|\leq (k+1+h)^{\beta+1}
\leq (\ln^3 (t/\rho_0) +1+h)^{\beta+1}.
\]
Hence, we have shown~\eqref{eq-thm-boundt}
for any $\alpha>3(\beta+1)$.

\medskip

To finish the proof of Theorem~\ref{th-couplingv1},
we argue why~\eqref{eq-thm-boundt}
implies~\eqref{eq-thm-corsim}. 
We use the result (apply Theorem~\ref{th-CC22} 
with $g(t)=e^{-t/2}$) 
that eventually $e^{\sqrt{\ln t}} \leq |\hW_t|$, so that
\[
\ln^\alpha t \leq \ln^{2\alpha} |\hW_t| =o(|\hW_t|).
\]
The asymptotic equivalence in~\eqref{eq-thm-corsim} 
therefore follows from~\eqref{eq-thm-boundt}.
$\hfill\Box$

\begin{rem}
For the sake of completeness, let us remark that 
it is possible to build stronger versions of the coupling.
\begin{itemize}
\item It is possible to build a version of the coupling 
where the coupling event holds with probability~$1$. 
In order to do so, we describe a way to ``patch'' 
the trajectories after a catastrophe, 
instead of simply letting them evolve independently. 
When a catastrophe occurs, 
treat both trajectories separately according 
to their respective construction procedures: 
each excursion may be accepted or discarded and this extends
(or not) processes~$\s$ and~$\hW$ by one excursion. 
At this point, the two processes may be lying at different levels.
To solve this, let the process with the smaller level 
run until it reaches the larger level. 
Then, the procedure can begin again,
starting from this common level. 

Due to catastrophes, the level sequences of~$\s$ and~$\hW$
do not necessarily coincide. 
This makes it more difficult to write it cleanly, 
which is why we chose not to present this almost sure 
construction in the article.  

It is straightforward to adapt the proof of 
Lemma~\ref{lem:expectationcatastrophe} in order to show that 
the number of catastrophes occurring 
during this construction has a finite expectation, 
hence this number is almost surely finite. 
What happens before the last catastrophe 
can be disregarded as it is eventually absorbed 
in the bound~$\ln^\alpha t$, 
so the proof of~\eqref{eq-thm-boundt} 
can be conducted without modifications. 

\item 
It is likely possible to build a version of the coupling satisfying
a stronger version of~\eqref{eq-thm-boundt}, 
obtained by removing the moduli there:
\begin{equation}
\label{eq-thm-boundtwomoduli}
 \big| \s_t-\hW_t \big|\leq \ln^\alpha t
    \quad \text{ eventually, as } t\to \infty.
\end{equation}
This, however, would require some finer control on the exit point
locations (of the $\s$- and $\hW$-excursions between the levels);
since~\eqref{eq-thm-boundt} is already enough for 
the purposes of this paper, we decided to 
leave proving~\eqref{eq-thm-boundtwomoduli}
as a possible topic for future research.
\end{itemize}
\end{rem}

\section{Proof of Theorem \ref{th-main}}
\label{s_proof_main}
As announced in the introduction, 
we now use Theorem~\ref{th-CC22} and the coupling of Theorem~\ref{th-couplingv1}
to prove Theorem~\ref{th-main}.

\begin{proof}[Proof of the first statement of Theorem \ref{th-main}]
Take $g$ satisfying the conditions of the theorem, 
we may assume without loss of generality that 
$(\ln t) g(\ln \ln t)$ goes to infinity 
as~$t$ goes to infinity. 
Take any $\eps \in (0,1)$ and consider the coupling $(\hW, \s)$ 
associated to~$\eps$. 
With probability not smaller than $1-\eps$, we have
\[
\frac{|\hW_t|}{2}\leq|\s_t|\leq2|\hW_t| \quad \text{ eventually}.
\]
Observe that $g/2$ and $2g$ satisfy the conditions of 
Theorem~\ref{th-CC22} and therefore 
we obtain on the event that the coupling is successful:
\begin{itemize}
 \item if $\int^\8 g=\infty$, 
  then for an unbounded set of~$t$'s:
  \begin{equation}
   M_{\lceil t \rceil}
   \leq |\s_{\lceil t \rceil}|
   \leq |\s_t|+1 \leq 2|\hW_t|+1 
   \leq 2e^{(\ln t) g(\ln \ln t)/2}+1,
  \end{equation} 
  which is eventually not larger than 
  $e^{(\ln \lceil t \rceil) g(\ln \ln \lceil t \rceil)}$;
 \item if $\int^\8 g<\infty$, then eventually:
 \begin{equation}
   |\s_n| \geq \frac{|\hW_n|}{2}
    \geq \frac{e^{2(\ln n) g(\ln \ln n)}}{2},
 \end{equation} 
  which is eventually not smaller than $e^{(\ln n) g(\ln \ln n)}$.
Since $n\mapsto e^{(\ln n) g(\ln \ln n)}$ is non-decreasing, 
it follows that $M_n\geq~e^{(\ln n) g(\ln \ln n)}$ eventually.
\end{itemize}
We have shown this to be true with probability 
at least~$1-\eps$, for any $\eps \in (0,1)$, 
so it is true almost surely. 
The first part of Theorem \ref{th-main} follows.
\end{proof}

\begin{proof}[Proof of the second statement of Theorem \ref{th-main}]
Consider some $\eps \in (0,1)$ and the 
associated coupling $(\s, \hW)$ 
having success probability not smaller than $1-\eps$. 
On this event, we obtain using $|\s_t| \sim |\hW_t|$ and Theorem~\ref{th-CC22}, that 
\[
  \limsup_{t \to \8} \frac{M_n}{\sqrt{n \ln \ln \ln  n}}=K 
\]
(we also made use of the fact that 
$|\s_t-\s_{\lfloor t \rfloor}|\leq 1$).
The latter is true with probability not smaller than $1-\eps$, 
for any $\eps \in (0,1)$, so it is true almost surely.
The second part of the theorem follows.
\end{proof}


\section*{Acknowledgements}
The authors were partially supported by
CMUP, member of LASI, which is financed by national funds
through FCT --- Funda\c{c}\~ao
para a Ci\^encia e a Tecnologia, I.P., 
under the project with reference UIDB/00144/2020.

\end{document}